\documentclass[reqno,10pt]{article}
\usepackage{graphicx}
\usepackage{amsmath}
\usepackage{amssymb}
\usepackage{amsthm}
\usepackage{enumitem}
\usepackage{bbm}
\usepackage[l2tabu,orthodox]{nag}
\usepackage[all,error]{onlyamsmath}
\usepackage[strict]{csquotes}
\usepackage{url}
\def\N {\mathbb{N}}

\def\H {\mathcal{H}}
\def\M {\mathcal{M}}

\def\P {\mathcal{P}}
\def\RR {\mathcal{R}}
\def\C {\mathbb{C}}
\def\R {\mathbb{R}}

\def\FF{\widehat{F}}
\def\E {\mathcal{E}}

\def\S{\mathbb{S}}

\newtheorem{theorem}{Theorem}
\newtheorem{lemma}[theorem]{Lemma}
\newtheorem{corollary}[theorem]{Corollary}
\newtheorem{prop}[theorem]{Proposition}
\theoremstyle{definition}
\newtheorem{definition}[theorem]{Definition}
\newtheorem{remark}[theorem]{Remark}

\numberwithin{equation}{section}
\numberwithin{theorem}{section}

\newenvironment{OMabstract}{\noindent\textbf{Abstract.} }{\medskip}
\newenvironment{OMsubjclass}{\noindent\textbf{Mathematics Subject Classification (2020):} }{\medskip}
\newenvironment{OMkeywords}{\noindent\textbf{Keywords:}  }{\medskip}

\begin{document}

\author{Mario Alberto Ruiz Caballero and Rafael del R\'io} 
\title{On Spectral Stability for Rank One Singular Perturbations}
\maketitle


\begin{OMabstract}
Embedded point spectra of rank one singular perturbations of an arbitrary self-adjoint operator $A$ on a Hilbert space $\H$ is studied. It is shown that these perturbations can be regarded as self-adjoint extensions of a densely defined closed symmetric operator $B$ with deficiency indices $(1,1)$. Assuming the deficiency vector of $B$ is cyclic for its self-adjoint extensions, we prove that the spectrum of $A$ contains a dense $\textit{G}_{\delta}$ subset where it is not possible to have eigenvalues for any rank one singular perturbation. Moreover, for a dense $\textit{G}_{\delta}$ set of rank one singular perturbations of $A$ their eigenvalues are isolated. The approach presented here unifies points of view
taken by different authors. 

   
\end{OMabstract}

\begin{OMkeywords}
    self-adjoint extension, rank one singular perturbation, embedded point spectra, singular continuous spectrum.
\end{OMkeywords}

\begin{OMsubjclass}
    47B02, 47B25, 47A55, 47A10.
\end{OMsubjclass}


\section{Introduction}  

A fundamental problem in spectral theory is to understand the behavior of spectra of self-adjoint operators when these operators are perturbed. One of the most natural perturbations are rank one regular perturbations, that is, perturbations of the form 
\begin{equation*}
A_{\alpha}=A+\alpha\langle\varphi,\cdot\rangle\varphi
\end{equation*}

where $\varphi$ is a cyclic vector for $A$ and the symbol $\langle\cdot,\cdot\rangle$ denotes inner product in $\H$. In particular it is known, that if $I$ is an interval contained in the spectrum of $A$, $\sigma(A)$, then it is possible for $A_{\alpha}$ to have dense point spectrum 
$\sigma_{p}(A_{\alpha})$ in $I$ for $a. e.$ $\alpha\in\R$ in Lebesgue sense, see \cite{SW}.
However, this cannot happen for every $\alpha\in\R$. It was shown in \cite{Gordon} and \cite{RMS} that there is a dense $\textit{G}_{\delta}$ set
$B\subset\R$ such that if $\alpha\in B$, then $\sigma_{p}(A_{\alpha})\cap I$ is empty and moreover, if $\alpha\in \R\setminus B$
then there is a dense $\textit{G}_{\delta}$ set $F\subset I$ such that $\sigma_{p}(A_{\alpha})\cap F$ is empty.
Nevertheless in some situations pure point spectrum is generic as shown in \cite{ArDa}.
In \cite{Rio} related problems were studied for Sturm-Liouville operators with local perturbations.\\


A natural question is whether similar results hold for rank one singular perturbations given by the formal expression
\begin{equation*}
A_{\alpha}=A+\alpha\langle\varphi,\cdot\rangle\varphi
\end{equation*}
where $\varphi\in\H_{-2}\setminus\H$. The symbol $\langle\cdot,\cdot\rangle$ denotes the duality pairing between $\H_{-2}$ and $\H_{2}$ or simply the action of linear functionals. Rank one singular perturbations are operators on the underlying Hilbert space whose domains are different from the domain of the unperturbed operator and the difference of their resolvents is a rank one bounded operator. In \cite{BSpaper} this question was considered for the case when $\varphi\in\H_{-1}\setminus\H$, i.e. for
so-called form bounded singular perturbations and $A$ being semi-bounded.\\

The case addressed in this paper includes the more general situation when form unbounded singular perturbations, i.e. $\varphi\in\H_{-2}\setminus\H_{-1}$ are considered. According to \cite{ALBKUR}, the difference between the two cases lies in the fact that if $\varphi\in\H_{-1}\setminus\H$, the formal expression $A_{\alpha}$ determines only one operator. On the other hand, if
$\varphi\in\H_{-2}\setminus\H_{-1}$, $A_{\alpha}$ turns out to be a biparametric family of operators $A_{\alpha,c}$ where each operator is determined by an extension $\varphi_{c}$ of $\varphi$. This theoretical framework will be explained later. Moreover, in the case $\varphi\in\H_{-1}\setminus\H$, the form-sum method is used while when $\varphi\in\H_{-2}\setminus\H_{-1}$ this method fails. An approach to describe rank one singular perturbations $A_{\alpha}$ is by means of the self-adjoint extensions of restriction $A$ to $Ker\varphi$, the subspace of $D(A)$ where $\varphi$ is equal to zero, which turns out to be a densely defined closed symmetric operator with deficiency indices $(1,1)$, see \cite{ALBKUR}. The main results in the current article are the following:

\begin{theorem}[Forbidden Energies]\label{maintheorem1}
Let $A$ be a self-adjoint operator on a Hilbert space $\H$. Let
$A_{\alpha,c}=A+\alpha\langle\varphi_{c},\cdot\rangle\varphi$ with $\varphi\in\H_{-2}\setminus\H$ and $c\in\R$ fixed. Let $\M$ be the cyclicity space of $(A-iI)^{-1}\varphi$ defined in the generalized sense. Then
\begin{equation*}
\Xi_{c}:=\left\lbrace x\in\sigma(A\upharpoonright_{\M}):x\not\in\sigma_{p}(A_{\alpha,c}\upharpoonright_{\M}),\;for\;any\;\alpha\in\R\setminus\left\lbrace 0\right\rbrace \right\rbrace 
\end{equation*}
contains a dense $\textit{G}_{\delta}$ subset in $\sigma(A\upharpoonright_{\M})$
\end{theorem}

\begin{theorem}[Forbidden Couplings]\label{maintheorem2}
Under the above assumptions
\begin{equation*}
\Gamma_{c}:=\left\lbrace \alpha\in\R\mid\sigma_{p}(A_{\alpha,c}\upharpoonright_{\M})\cap\sigma(A\upharpoonright_{\M})=\varnothing\right\rbrace
\end{equation*}
contains a dense $\textit{G}_{\delta}$ subset in $\R$.
\end{theorem}

The present paper essentially unifies the methods of \cite{RMS} and \cite{Gordon} in the framework of self-adjoint extensions. Following the approach of \cite{RMS} we get a theorem concerning forbidden energies for self-adjoint extensions. On the other hand the ideas of \cite{Gordon} lead to a theorem on forbidden extension parameters and
allow to show that Theorem \ref{maintheorem1} is equivalent to Theorem \ref{maintheorem2}.\\

The paper is divided as follows. In Section 2 both the von Neumann's Extensions Theory and the theoretical framework given in \cite{ALBKUR} for rank one singular perturbations are provided. In Section 3 some results of \cite{RMS} originally proved for Borel-Stieltjes transforms are extended for Nevanlinna-Herglotz functions. In Section 4 in
order to illustrate what happens when the spectrum of self-adjoint extensions is not simple, a version of
the well-known theorem from Aronszajn-Donoghue Theory on characterization of eigenvalues by improper integrals when
self-adjoint extensions are reduced to a cyclicity space is shown. Then with this result Theorem \ref{maintheorem1} is
proven. In Section 5 a theorem on forbidden extension parameters, that is, for self-adjoint extensions of a densely defined closed symmetric operator with deficiency indices $(1,1)$ is obtained and Theorem \ref{maintheorem2} is deduced by transforming rank one singular perturbations into self-adjoint extensions of a certain symmetric operator through a homeomorphism. As a consequence of these results, for a dense $\textit{G}_{\delta}$ set either of rank one singular perturbations of $A$ or self-adjoint extensions of a symmetric operator, their eigenvalues are isolated.
Moreover, if we assume empty absolutely continuous spectrum, there is pure singular continuous spectrum for this set
of operators.

\section{Preliminaries}

\subsection{Self-adjoint extensions}

We recall von Neumann's Extension Theorem of Symmetric Operators. For this, the following definition is given.

\begin{definition}[Definition 2.2, \cite{Sch} and Equation 7.1.44, \cite{BS}]
Let $B$ be a densely defined closed symmetric operator on  $\H$. We shall call deficiency spaces of $B$ to the sets

\begin{equation*}
K_{\pm}(B) :=Ran(B\pm iI)^{\perp}=Ker(B^{*}\mp iI)
\end{equation*}

where $\perp$ and $B^{*}$ denote orthogonal complement in $\H$ and adjoint operator. Also, we shall call deficiency indices of $B$ to the pair $\left( d_{+}(B), d_{-}(B)\right) $ where
\begin{equation*}
d_{\pm}(B):=dim\;K_{\pm}(B).
\end{equation*}
\end{definition}

Let $\mathcal{B}(B)$ be the set of closed symmetric extensions of $B$ and $\mathcal{V}(B)$ the set of partial isometries
from $K_{+}(B)$ to $K_{-}(B)$. We state the next theorem.

 \begin{theorem}[Theorem 13.9, \cite{Sch} and Theorem 7.4.1, \cite{BS}]\label{teovonneum}
Let $B$ be a densely defined closed symmetric operator on $\H$. There exists a bijective
mapping from $\mathcal{V}(B)$ to $\mathcal{B}(B)$ given by
\begin{equation*}
V\mapsto T_{V}:=B^{*}\upharpoonright_{D(T_{V})}
\end{equation*}
where
	\begin{equation*}
	D(T_{V})=D(B)\dotplus (I+V)D(V).
	\end{equation*}
Furthermore, $T_{V}$ is self-adjoint if and only if $V$ is unitary from $K_{+}(B)$ to $K_{-}(B)$.
\end{theorem}
 
Suppose that $B$ in the above theorem has deficiency indices $(1,1)$ and $u_{\pm}\in K_{\pm}(B)$ is a generating vector with norm equal to $1$. The vector $u_{+}$ is called deficiency vector. For each
$\theta\in\left[ 0,\pi\right)$ one defines the operator $V_{\theta}:K_{+} (B)\longrightarrow K_{-}(B)$ given by
$V_{ \theta}(u_{+}):=e^{-2i\theta}u_{-}$. Denote the self-adjoint extensions of $B$ given by Theorem \ref{teovonneum} as $T_{\theta}$ , with $\theta\in\left[ 0,\pi\right) $, where

\begin{equation}\label{tesubteta}
D(T_{\theta})=D(B)\dotplus span\left\lbrace u_{+}+e^{-2i\theta}u_{-}\right\rbrace 
\end{equation}
and
\begin{center}
$T_{\theta}(x+cu_{+}+ce^{-2i\theta}u_{-})=Bx+ciu_{+}-cie^{-2i\theta}u_{-}$, $x\in D(B)$, $c\in\C$.
\end{center} 
 
 Denote by $\M$ the cyclicity space of $u_{+}$ for any $T_{\theta}$ which by definition is

\begin{equation*}
\M:=\overline{span\left\lbrace (T_{\theta}-zI)^{-1}u_{+}:z\in\C\setminus\R\right\rbrace }.
\end{equation*}
\begin{remark}
We know that $\M$ does not depend on $\theta$ and is a reducing subspace for $T_{\theta}$, for all
$\theta\in\left[ 0,\pi\right)$ (see \cite[Section 2]{DON}, \cite[Lemma 4.5]{LIVSIC}). Therefore, one has the
restrictions $T_{\theta}\upharpoonright_{\M}$ acting on the Hilbert space $\M$ with domain
\begin{equation*}
D(T_{\theta}\upharpoonright_{\M}):=D(T_{\theta})\cap\M
\end{equation*}
which are self-adjoint operators and have simple spectrum since by definition $u_{+}$ is cyclic for $T_{\theta}\upharpoonright_{\M}$.
\end{remark}

\subsection{Rank one singular perturbations}

Let $A$ be a self-adjoint operator on a Hilbert space $\H$. Consider the $A$-scale of Hilbert spaces
\begin{equation*}
\H_{s}\subseteq\H\subseteq\H_{-s}
\end{equation*}
where $\H_{s}:=\left( D(\vert A\vert^{\frac{s}{2}}),\Vert\cdot\Vert_{s}\right)$ with $\Vert x\Vert_{s}:=\Vert (\vert A\vert +I)^{\frac{s}{2}}x\Vert$
for all $s\geqslant 0$ and $\H_{-s}$ is the completion of $\H$ with the norm $\Vert\cdot\Vert_{-s}$, i.e. the
space of linear functionals with its usual norm $\left( \H_{s}^{*},\Vert\cdot\Vert_{\H_{s}^{*}}\right)$. Given $\varphi\in\H_{-2}\setminus\H$ and $\alpha\in\R$, rank one singular perturbations of $A$ are defined by the formal expression
\begin{equation}\label{Asubalfa}
A_{\alpha}=A+\alpha\langle\varphi,\cdot\rangle\varphi
\end{equation}
where $\langle\cdot,\cdot\rangle$ denotes the duality pairing between $\H_{-2}$ and $\H_{2}$
or simply the action of linear functionals. Our goal is to obtain self-adjoint realizations on $\H$ of the expressions $A_{\alpha}$, which will be self-adjoint extensions of a densely defined closed symmetric operator with deficiency indices $(1,1)$. The next result is the key to convert the family of formal expressions (\ref{Asubalfa}) into self-adjoint extensions.

\begin{theorem}[Lemma 1.2.3, \cite{ALBKUR}]\label{Apunto}
Let $A$ be a self-adjoint operator on $\H$ and $\varphi\in\H_{-2}\setminus\H$. Then 
\begin{center}
$\dot{A}:=A\upharpoonright_{D(\dot{A})}$ where
$D(\dot{A}):=\left\lbrace x\in D(A):\langle\varphi,x\rangle =0\right\rbrace $ 
\end{center}
is a densely defined closed symmetric operator with deficiency indices $(1,1)$.
\end{theorem}

We briefly present the approach of \cite{ALBKUR}. Consider the operator 
\begin{equation}\label{opergeneralsen}
(A\pm iI)^{-1}:\H_{s-2}\longrightarrow\H_{s},\;s=0,1
\end{equation}
in the generalized sense, i.e. for $\phi\in\H_{s-2}$ and $\eta\in\H_{s}$
\begin{equation*}
\langle\phi, (A\mp iI)^{-1}\eta\rangle = \langle(A\pm iI)^{-1}\phi,\eta\rangle.
\end{equation*}

By the above theorem and using the first formula of von Neumann (see \cite[Theorem 8.11]{Weid} and \cite[Theorem 7.1.11]{BS}), it turns out that
\begin{equation}\label{neumanformula}
D(\dot{A}^{*})=D(\dot{A})\dotplus span\left\lbrace g_{i},g_{-i}\right\rbrace
\end{equation}
where
\begin{equation*}
g_{\pm i}:=(A\mp i)^{-1}\varphi.
\end{equation*}
are the deficiency vectors for $\dot{A}$. We get the family of self-adjoint extensions of $\dot{A}$ given by Theorem \ref{teovonneum} as $A(v)$, where $v\in\S^{1}$, set of unimodular complex numbers,  such that

\begin{equation}
D(A(v))=\left\lbrace x+a_{+}g_{i}+a_{-}g_{-i}\in D(\dot{A}^{*}):a_{-}=-\overline{v} a_{+}\right\rbrace .
\end{equation}

One has that
\begin{equation*}
A(A^{2}+I)^{-1}\varphi=\frac{1}{2}\left[ (A-iI)^{-1}\varphi+(A+iI)^{-1}\varphi\right] \in\H.
\end{equation*}
So we can write (\ref{neumanformula}) of the next form

\begin{equation}\label{defsumdosspace}
D(\dot{A}^{*})=D(A)\dotplus span\left\lbrace A(A^{2}+I)^{-1}\varphi\right\rbrace .
\end{equation}

This makes another family of self-adjoint extensions $A^{\gamma}$ of $\dot{A}$, with $\gamma\in\R\cup\left\lbrace\infty \right\rbrace $, given by
\begin{equation*}
D(A^{\gamma})=\left\lbrace x+bA(A^{2}+I)^{-1}\varphi\in D(\dot{A}^{*}):\varphi(x)=\gamma b\right\rbrace .
\end{equation*} 

Both families of extensions are related by the formula
\begin{equation}\label{v}
v=\frac{\gamma +i}{\gamma -i}.
\end{equation}

By (\ref{opergeneralsen}), the following is remarked
\begin{itemize}
\item[i)] If $\varphi\in\H_{-1}\setminus\H$, $\langle\varphi,(A-zI)^{-1}\varphi\rangle$ exists because
$(A-zI)^{-1}\varphi\in\H_{1}$.
\item[ii)] If $\varphi\in\H_{-2}\setminus\H_{-1}$, $\langle\varphi,(A-zI)^{-1}\varphi\rangle$ is not well-defined
since $(A-zI)^{-1}\varphi\in\H$ but in general $(A-zI)^{-1}\varphi\not\in\H_{2}$.
\end{itemize}
In case ii) the functional $\varphi$ cannot be extended to the space $D(\dot{A}^{*})$ given by (\ref{defsumdosspace}). So we must renormalize the expression
\begin{center}
$\langle\varphi, A(A^{2}+I)^{-1}\varphi\rangle$.
\end{center}

For all $c\in\R$, the linear functional $\varphi_{c}:D(\dot{A}^{*})\longrightarrow\C$ is given by
\begin{center}
$\langle\varphi_{c}, x+bA(A^{2}+I)^{-1}\varphi\rangle =\langle\varphi ,x\rangle+bc,\;x\in D(A),\;b\in\C$.
\end{center}

In order to define $A_{\alpha}$ as a restriction of $\dot{A}^{*}$ we need to extend the functional $\varphi$. The only extensions of $\varphi$ to $D(\dot{A}^{*})$ are given by $\varphi_{c}$ (see \cite[Lemma 1.3.1]{ALBKUR}). Thus one has as many self-adjoint realizations of $A_{\alpha}$ as extensions $\varphi_{c}$. Therefore,
\begin{equation*}
A_{\alpha}=\left\lbrace A_{\alpha ,c}\right\rbrace_{c\in\R}
\end{equation*}
where
\begin{equation}\label{alface}
A_{\alpha ,c}=A+\langle\varphi_{c},\cdot\rangle\varphi
\end{equation}

\begin{remark}
In case i), the unique extension $\varphi_{c}$ is given by 
\begin{equation*}
c=\langle\varphi,(A-zI)^{-1}\varphi\rangle.
\end{equation*}
Therefore, $A_{\alpha}$ is uniquely defined.
\end{remark}

We describe self-adjoint realizations of $A_{\alpha ,c}$ in terms of the operators (\ref{tesubteta}). Consider
that $arg$ function is valued in $\left[ 0,2\pi\right)$.

\begin{theorem}\label{ultimate}
Let $\varphi\in\H_{-2}\setminus\H$. Consider $A_{\alpha ,c}$ given by (\ref{alface}), $\dot{A}$ as in Theorem \ref{Apunto} and $T_{\theta}$ given by (\ref{tesubteta}) with $B=\dot{A}$. Then $A_{\alpha ,c}=T_{\theta}$ where
\begin{equation}\label{tetace}
\theta=\frac{1}{2}arg\left[ -\frac{1+\alpha(c-i)}{1+\alpha(c+i)}\right] .
\end{equation}
\end{theorem}
\begin{proof}
Both by \cite[Theorem 1.3.1]{ALBKUR} and by \cite[Theorem 1.3.2]{ALBKUR}, one has that $A_{\alpha ,c}=A^{\gamma}$
where
\begin{equation*}
\gamma=-\left( \frac{1}{\alpha}+c\right) .
\end{equation*}
Substituting the above formula in (\ref{v}) we obtain that
\begin{equation}\label{vform}
v=\frac{1+\alpha(c-i)}{1+\alpha(c+i)}
\end{equation}
of which $A_{\alpha ,c}=A(v)$. On the other hand, given $\theta\in\left[ 0,\pi\right) $ and $v\in\S^{1}$,
$T_{\theta}=A(v)$ iff

\begin{equation*}
\theta =\frac{1}{2}arg(-v)
\end{equation*}
Substituting formula (\ref{vform}) in the last expression we obtain (\ref{tetace}) and $A_{\alpha ,c}=T_{\theta}$.
\end{proof}

\begin{remark}
$A=A_{0 ,c}=T_{\frac{\pi}{2}}=A(1)=A^{\infty}$.
\end{remark}

\section{Scalar Nevanlinna-Herglotz functions}

We start with essential facts on Nevalinna-Herglotz functions. For positive Borel measures $\mu$ such that 
\begin{equation}\label{grande}
\int \frac{d\mu(x)}{1+x^{2}}<\infty 
\end{equation}
it is defined the function $F_{\mu}:\C^{+}\longrightarrow\C^{+} $, where $\C^{+}$ is the complex upper half-plane, given by
\begin{equation*}
F_{\mu}(z):=\int \left( \frac{1}{x-z}-\frac{x}{1+x^{2}}\right) d\mu(x).
\end{equation*}
By Canonical Integral Representation of Nevanlinna-Herglotz functions (see \cite[Theorem 2.2(iii)]{GT}, \cite[Theorem F.1]{Sch}), $F_{\mu}$ is a Nevanlinna-Herglotz function. We show some properties for these functions. The following proposition is a extension of \cite[Theorem 1.2(iii)]{BSpaper}. Although it is a well-known fact, for the convenience
of the reader a proof is offered. 

\begin{prop}\label{implicacion}
Let $\mu$ be a positive Borel measure satisfying (\ref{grande}). Suppose $y\in\R$ holds
\begin{equation*}
\int_{\R} \frac{d\mu(x)}{(x-y)^{2}}<\infty.
\end{equation*}
Then 
\begin{equation*}
F_{\mu}(y+i0):=\lim_{\varepsilon\rightarrow 0}F_{\mu}(y+i\varepsilon)
\end{equation*}
exists and is real.
\end{prop}
\begin{proof}
Let $d\rho(x):=\frac{d\mu(x)}{1+x^{2}}$ be a finite measure. Therefore there exists the function
\begin{equation*}
J(z):=\int_{\R} \frac{1}{x-z} d\rho(x).
\end{equation*}
Then
\begin{align*}
F_{\mu}(z)&=\int_{\R}  \frac{1+zx}{x-z} d\rho(x)\\
		&=\int_{\R}  \frac{zx-z^{2}}{x-z} d\rho(x)+\int_{\R}  \frac{1+z^{2}}{x-z} d\rho(x)\\
		&= z\rho(\R)+(1+z^{2})J(z).
\end{align*}

Furthermore
\begin{equation}\label{desigualdadmia}
\int_{\R} \frac{d\rho(x)}{(x-y)^{2}}\leqslant\int_{\R} \frac{(1+x^{2})d\rho(x)}{(x-y)^{2}}
=\int_{\R} \frac{d\mu(x)}{(x-y)^{2}}<\infty.
\end{equation}
By \cite[Theorem 1.2(iii)]{BSpaper},
\begin{equation*}
J(y+i0):=\lim_{\varepsilon\rightarrow 0}J(y+i\varepsilon)
\end{equation*}
exists and is real. Thus, we have concluded.
\end{proof}

The next lemma appears in the proof of \cite[Theorem 2.1]{RMS} for the case of Borel-Stieltjes transformations.

\begin{lemma}\label{lemaG}
Let $\mu$ be a positive Borel measure satisfying (\ref{grande}). Given $\lambda\in\R$, the functions
\begin{equation*}
G_{n}(\lambda):=\int_{\R} \frac{d\mu(x)}{(x-\lambda)^{2}+\frac{1}{n^{2}}}
\end{equation*}
are continuous and
\begin{equation*}
\int_{\R}\frac{d\mu(x)}{(x-\lambda)^{2}}=\lim_{n\rightarrow\infty}G_{n}(\lambda).
\end{equation*}
\end{lemma}
\begin{proof}
By doing some calculations we find that
\begin{equation*}
G_{n}(\lambda)=nImF_{\mu}\left( \lambda +i\frac{1}{n}\right).
\end{equation*}
Since the function on the right hand side is continuous with $n$ fixed, so is $G_{n}$ for all $n\in\N$. By Monotonous Convergence Theorem, the second holds.
\end{proof}

We provide the following definitions.

\begin{definition}
\begin{itemize}
\item Let $X$ be a metric space. A subset $M\subseteq X$ is $\textit{G}_{\delta}$ in $X$ if there is some countable family $\left\lbrace M_{i} \right\rbrace_{i\in\N}$ of open subsets in $X$ such that $M=\bigcap_{i\in\N}M_{i}$.
\item A subset $S\subseteq\R$ is called a support of a Borel measure $\mu$ if $\mu\left( \R\setminus S\right)=0$.
\item The smallest closed support of $\mu$ is called the topological support of $\mu$ and denoted by $supp(\mu)$.
\end{itemize}
\end{definition} 

Due to the previous results, a generalization of \cite[Theorem 2.1]{RMS} for a larger class of measures is proven.

\begin{prop}\label{teo2.1}
Let $\mu$ such that (\ref{grande}) holds. Then
\begin{equation}\label{conjuntomio}
\left\lbrace y\in supp(\mu):\int_{\R} \frac{d\mu(x)}{(x-y)^{2}}=\infty\right\rbrace 
\end{equation}
is dense $\textit{G}_{\delta}$ in $supp(\mu)$.
\end{prop}
\begin{proof}
Let $d\rho(x):=\frac{d\mu(x)}{1+x^{2}}$ be a finite measure and
\begin{equation*}
A:=\left\lbrace y\in\R:\int_{\R} \frac{d\rho(x)}{(x-y)^{2}}=\infty\right\rbrace.
\end{equation*}
By (\ref{desigualdadmia}),
\begin{equation*}
A\subseteq\left\lbrace y\in\R:\int_{\R} \frac{d\mu(x)}{(x-y)^{2}}=\infty\right\rbrace.
\end{equation*}

Due to that $\rho$ and $\mu$ are equivalent we have
\begin{equation*}
A\cap supp(\rho)\subseteq supp(\mu)\cap\left\lbrace y\in\R:\int_{\R} \frac{d\mu(x)}{(x-y)^{2}}=\infty\right\rbrace.
\end{equation*}
By \cite[Theorem 2.1]{RMS}, $A$ is dense in $supp(\rho)$ and hence the set (\ref{conjuntomio}) is dense
in $supp(\mu)$.\\

By continuity of $G_{n}$ according to Lemma \ref{lemaG} and since (\ref{conjuntomio}) turns out to be
\begin{equation*}
\bigcap_{m\in\N}\bigcup_{n\in\N}\left\lbrace y\in supp(\mu):\int_{\R}\frac{d\mu(x)}{(x-y)^{2}+\frac{1}{n^{2}}}>m\right\rbrace .
\end{equation*}
It follows that (\ref{conjuntomio}) is $\textit{G}_{\delta}$ in $supp(\mu)$.
\end{proof}

We conclude with the following corollary.

\begin{corollary}\label{teo2.1--}
Let $\mu$ such that (\ref{grande}) holds. Then
\begin{equation*}
supp(\mu)\bigcap\left\lbrace y\in\R:\int_{\R}\frac{d\mu(x)}{(x-y)^{2}}<\infty\right\rbrace 
\end{equation*}
is a countable union of closed nowhere dense subsets in $supp(\mu)$.
\end{corollary}

Our goal in the following sections will be to obtain results on \textbf{forbidden energies} and \textbf{forbidden extension parameters} for self-adjoint extensions $T_{\theta}\upharpoonright_{\M}$ and after for rank one singular perturbations $A_{\alpha,c}$.

\section{Forbidden Energies}

We extend \cite[Theorem 4]{DON}, classical in the Aronszajn-Donoghue Theory, for the case when $u_{+}$ is not
cyclic. For this, the spectral family $\E^{0}$ of $T_{\theta_{0}}\upharpoonright_{\M}$ is considered. Define the measure $\mu^{0}$ such that $d\mu^{0}(x):=(1+x^{2})d\langle u_{+},\E^{0} (x)u_{+}\rangle$. We denote by $\sigma_{p}$
the set of eigenvalues.

\begin{theorem}\label{teo1.1}
For each $\theta\neq \theta_{0}$,
\begin{equation*}
\sigma_{p}(T_{\theta}\upharpoonright_{\M})=\left\lbrace y\in\R :
\int_{\R}\frac{d\mu^{0}(x)}{(x-y)^{2}}<\infty , F_{\mu^{0}}(y+i0)=cot(\theta - \theta_{0})\right\rbrace.
\end{equation*}
\end{theorem}
\begin{proof}
From \cite[Section 4]{GT} one has that $B\upharpoonright_{\M}$ is densely defined closed symmetric with deficiency indices $(1,1)$ on $\H_ {0}$. Let $R_{\theta}$ be its self-adjoint extensions. By definition, $u_{+}$ is cyclic for every $R_{\theta}$. Let $\E'$ be the spectral family of $R_{\theta_{0}}$ and $d\mu'(x):=(1+x^{2})d\langle u_{+},\E' (x)u_{+}\rangle$. By \cite[Theorem 4]{DON},
\begin{equation*}
\sigma_{p}(R_{\theta})=\left\lbrace y\in\R :
\int_{\R}\frac{d\mu'(x)}{(x-y)^{2}}<\infty , F_{\mu'}(y+i0)=cot(\theta - \theta_{0})\right\rbrace .
\end{equation*}
It is required to prove that $R_{\theta}=T_{\theta}\upharpoonright_{\M}$. Let us first note that
\begin{equation*}
K_{\pm}(B)=Ran(B\pm iI)^{\perp}\subseteq Ran(B\upharpoonright_{\M}\pm iI)^{\perp}=K_{\pm}(B\upharpoonright_{\M}).
\end{equation*}
Both $B$ and $B\upharpoonright_{\M}$ have deficiency indices $(1,1)$, therefore
$K_{\pm}(B)=K_{\pm}(B\upharpoonright_{\M})$. Let us show that
$D(B^{*}\upharpoonright_{\M})=D\left[ (B\upharpoonright_{\M})^{*}\right] $.\\

If $x\in D(B^{*}\upharpoonright_{\M})$, then $x=f+a_{+}+a_{-}\in\M$ with $f\in D(B)$ and $a_{\pm}\in K_{\pm}(B)$.
By the above, $a_{\pm}\in K_{\pm}(B\upharpoonright_{\M})$ and since $ K_{\pm}(B\upharpoonright_{\M})\subseteq \M$ one has that $f\in D(B)\cap \M$. Then,
$x\in D\left[ (B\upharpoonright_{\M})^{*}\right]$.\\

If $x\in D\left[ (B\upharpoonright_{\M})^{*}\right]$, then $x=g+b_{+}+b_{-}$ with $g\in D(B\upharpoonright_{\M})=D(B)\cap \M$ and $b_{\pm}\in K_{\pm}(B\upharpoonright_{\M})\subseteq\M$. We conclude that $x\in D(B^{*}\upharpoonright_{\M})$.\\

Thus, $B^{*}\upharpoonright_{\M}=(B\upharpoonright_{\M})^{*}$. Consider, for all $\theta\in\left[ 0,\pi\right)$, unitary operators
\begin{equation*}
V_{\theta}:K_{+}(B)=K_{+}(B\upharpoonright_{\M})\longrightarrow K_{-}(B)=K_{-}(B\upharpoonright_{\M}).\end{equation*}
Then

\begin{align*}
D(R_{\theta})&= D(B\upharpoonright_{\M})\dotplus K_{+}(B\upharpoonright_{\M})\dotplus V_{\theta}\left[  K_{+}(B\upharpoonright_{\M})\right] \\
			&=\left[ D(B)\dotplus K_{+}(B)\dotplus V_{\theta}\left(  K_{+}(B)\right)\right] \cap \M \\
			&= D(T_{\theta}\upharpoonright_{\M}).
\end{align*}

Since in particular $R_{\theta_{0}}=T_{\theta_{0}}\upharpoonright_{\M}$, we have that $\E'=\E^{0}$. Therefore,
$\mu'=\mu^{0}$ and $F_{\mu'}=F_{c}$. In conclusion, we have the desired result.
\end{proof}

We obtain this corollary.

\begin{corollary}\label{corolario00}
Consider $\mu^{0}$ as above. Then
\begin{equation*}
\bigcup_{\theta\in\left[ 0,\pi\right)\setminus\lbrace \theta_{0}\rbrace} \sigma_{p}(T_{\theta}\upharpoonright_{\M})=
\left\lbrace y\in\R:\int_{\R}\frac{d\mu^{0}(x)}{(x-y)^{2}}<\infty\right\rbrace .
\end{equation*}
\end{corollary}
\begin{proof}
By the previous theorem, if $y\in\sigma_{p}(T_{\theta}\upharpoonright_{\M})$, for one $\theta\in\left[ 0,\pi\right)\setminus\lbrace \theta_{0}\rbrace$, then $\int_{\R}\frac{d\mu^{0}(x)}{(x-y)^{2}}<\infty$. On the other hand, suppose
 $\int_{\R}\frac{d\mu^{0}(x)}{(x-y)^{2}}<\infty$. By Proposition \ref{implicacion}, $F_{\mu^{0}}(y+i0)\in\R$.
Furthermore 
\begin{equation*}
h:\left[ 0,\pi\right)\setminus\left\lbrace \theta_{0} \right\rbrace\longrightarrow\R\;where\;h(\theta):=cot(\theta-\theta_{0})
\end{equation*}
is a bijection. Thus, there exists $\theta\in\left[ 0,\pi\right)\setminus\left\lbrace \theta_{0} \right\rbrace $ such that $F_{\mu^{0}}( y+i0)=h(\theta)$. By the previous theorem, $y\in \sigma_{p}(T_{\theta}\upharpoonright_{\M})$.
\end{proof}

According to above, we show the following corollary.

\begin{corollary}\label{prop1}
The set
\begin{equation}\label{prop1set}
\sigma(T_{\theta_{0}}\upharpoonright_{\M})\cap\bigcup_{\theta\in\left[ 0,\pi\right)\setminus\lbrace \theta_{0}\rbrace} \sigma_{p}(T_{\theta}\upharpoonright_{\M})
\end{equation}
is a countable union of closed nowhere dense subsets in $\sigma(T_{\theta_{0}}\upharpoonright_{\M})$
(and therefore in $\R$).
\end{corollary}
\begin{proof}
Since $supp(\mu^{0})=\sigma(T_{\theta_{0}}\upharpoonright_{\M})$, by Corollary \ref{teo2.1--} and
Corollary \ref{corolario00} one concludes.
\end{proof}

We inmediately conclude the following theorem.

\begin{theorem}\label{zzz}
Let $\theta_{0}$ fixed. Then the set of points in $\sigma(T_{\theta_{0}}\upharpoonright_{\M})$ which are not eigenvalues
for any $T_{\theta}\upharpoonright_{\M}$ with $\theta\neq\theta_{0}$ is dense $\textit{G}_{\delta}$ in
$\sigma(T_{\theta_{0}}\upharpoonright_{\M})$.
\end{theorem}
\begin{proof}
This follows by the fact that the set of points in $\sigma(T_{\theta_{0}}\upharpoonright_{\M})$ which are not eigenvalues
for any $T_{\theta}\upharpoonright_{\M}$ with $\theta\neq\theta_{0}$ turns out to be the complement in
$\sigma(T_{\theta_{0}}\upharpoonright_{\M})$ of (\ref{prop1set}).
\end{proof}

This result leads to the proof of the first main theorem.

\begin{proof}[Proof of Theorem \ref{maintheorem1}]
By Theorem \ref{ultimate} it turns out that
\begin{align*}
\Xi_{c}&=\left\lbrace x\in\sigma(T_{\frac{\pi}{2}}\upharpoonright_{\M}):x\not\in\sigma_{p}(T_{\theta}\upharpoonright_{\M}),\;for\;any\;\alpha\in\R\setminus\left\lbrace 0\right\rbrace\;where\;\theta =f_{c}(\alpha)\right\rbrace \\
&=\left\lbrace x\in\sigma(T_{\frac{\pi}{2}}\upharpoonright_{\M}):x\not\in\sigma_{p}(T_{\theta}\upharpoonright_{\M}),\;for\;any\;\theta\in\left[ 0,\pi\right)\setminus\left\lbrace\frac{\pi}{2},f_{c}(\infty)\right\rbrace\right\rbrace \\
&\supseteq \left\lbrace x\in\sigma(T_{\frac{\pi}{2}}\upharpoonright_{\M}):x\not\in\sigma_{p}(T_{\theta}\upharpoonright_{\M}),\;for\;any\;\theta\in\left[ 0,\pi\right)\setminus\left\lbrace\frac{\pi}{2}\right\rbrace\right\rbrace .
\end{align*}
where
\begin{equation*}
f_{c}(\alpha):=\frac{1}{2}arg\left[ -\frac{1+\alpha(c-i)}{1+\alpha(c+i)}\right].
\end{equation*}
Therefore Theorem \ref{maintheorem1} follows from Theorem \ref{zzz} as $\theta_{0}=\frac{\pi}{2}$.
\end{proof}

\section{Forbidden Extension Parameters and Couplings}
Let us start proving the next lemma.

\begin{lemma}\label{lem1}
Suppose that $\theta\in\left[ 0,\pi\right)$ and $E\in\R$. If $y\in \left( Ker(T_{\theta}-EI)\setminus\left\lbrace 0\right\rbrace\right) \cap\M$,
then $\langle y,u_{+}\rangle\neq 0$.
\end{lemma}
\begin{proof}
Suppose there exists $y\in \left( Ker(T_{\theta}-EI)\setminus\left\lbrace 0\right\rbrace\right) \cap\M$ such that $\langle y,u_{+}\rangle =0$. Given $z\in\C\setminus\R$
\begin{equation*}
(T_{\theta}-\overline{z}I)^{-1}y=(E-\overline{z})^{-1}y.
\end{equation*}
Then
\begin{equation*}
\langle y,(T_{\theta}-zI)^{-1}u_{+}\rangle =
\langle (T_{\theta}-\overline{z}I)^{-1}y,u_{+}\rangle =\langle (E-\overline{z})^{-1}y,u_{+}\rangle =0.
\end{equation*}
Since $u_{+}$ is cyclic for $T_{\theta}\upharpoonright_{\M}$ when $\theta$ is fixed, one concludes $y=0$.
\end{proof}

\begin{definition}
Let $X$ be a Banach space and $X^{*}$ be its dual space. The weak topology is the weakest topology in $X$ such that each functional in $X^{*}$ is continuous. The weak*-topology is the weakest topology in $X^{*}$ such that each functional in $X^{**}$ is continuous.
\end{definition}

\begin{remark}
If $X$ is a Hilbert space the weak topology and weak*-topology coincide. Therefore by Banach-Alaoglu-Bourbaki Theorem
the closed balls in a Hilbert space are compact with respect to the weak topology. 
\end{remark}

Let $\tau :=\left[ 0,\pi \right] \times\R\times\M$ where the Hilbert space $\M$ will be endowed with the weak topology. By the last lemma, we can define the following sets:

\begin{itemize}
\item $\tau_{M}:=\left[ 0,\pi \right]\times\R\times B_{M}\cap\M$ where $B_{M}$ is the closed ball in $\H$ with
center at $0$ and radius $M$
\item $Q_{M}:=\left\lbrace \left(\theta,E,y\right)\in\tau_{M}:y\in Ker(T_{\theta}-EI)\;such\;that\;\langle y,u_{+}\rangle =1 \right\rbrace $.
\end{itemize}

\begin{remark}
The topological space $\tau_{M}$ is metrizable because $B_{M}\cap\M$ is too. It is due to the separability of $\M$. Further $B_{M}\cap\M$ is a convex set in $\M$ so that is strongly, weakly and weakly sequentially
closed in $\M$. Therefore $\tau_{M}$ is a closed subspace of $\tau$.
\end{remark}

We propose the next definition.

\begin{definition}
Let $A:\H\longrightarrow\H$ be an operator. It is said to be weakly closed in $\H$ if given
$\left( x_{n}\right)_{n\in\N}\subseteq D(A)$ such that
\begin{center}
$x_{n}\longrightarrow^{w}x\in\H$ and $Ax_{n}\longrightarrow^{w}y\in\H$
\end{center}
then $x\in D(A)$ and $Ax=y$.
\end{definition}

The proof of the following proposition is like the classical case. Only the fact that the inner product is continuous with respect to weak limits is required .

\begin{prop}\label{debilcerrado}
Let $A:\H\longrightarrow\H$ be a densely defined operator. Then $A^{*}$ is weakly closed on $\H$.
\end{prop}

We prove the following lemma.

\begin{lemma}\label{lem2}
The set $Q_{M}$ is closed in $\tau_{M}$.
\end{lemma}
\begin{proof}
Let $(\theta_{n},E_{n},y_{n})\in Q_{M}$ be a sequence such that it converges to $(\theta ,E, y)\in\tau_{M}$. We are going to show $(\theta ,E, y)\in Q_{M}$. By definition for every $n\in\N$, $y_{n}\in  Ker(T_{\theta_{n}}-E_{n}I)\cap\M$
such that $\langle y_{n},u_{+}\rangle =1$. Since $y_{n}\longrightarrow^{w} y$
one has $\langle y,u_{+}\rangle =1$ and hence $y\neq 0$. Moreover, 
\begin{equation*}
B^{*}y_{n}=T_{\theta_{n}}y_{n}=E_{n}y_{n}\longrightarrow^{w}Ey.
\end{equation*}
By Proposition \ref{debilcerrado}, $y\in D(B^{*})$ and $B^{*}y=Ey$. That is, there exist $x\in D(B)$ and
$a,b\in\C$ such that
$y=x+au_{+}+bu_{-}$. Then, for each $n\in\N$ there exist $x_{n}\in D(B)$ and $c_{n}\in\C$  such that
\begin{equation*}
y_{n}=x_{n}+c_{n}u_{+}+c_{n}e^{-2i\theta_{n}}u_{-}\longrightarrow^{w}y=x+au_{+}+bu_{-}.
\end{equation*}
On the other hand, by using the inner product of the graph of $B^{*}$
\begin{align*}
\langle y_{n},u_{+}\rangle_{B^{*}} :=\langle y_{n},u_{+}\rangle +\langle B^{*}y_{n},B^{*}u_{+}\rangle
									&=\langle y_{n},u_{+}\rangle +\langle E_{n}y_{n},iu_{+}\rangle \\
									&=\langle y_{n},u_{+}\rangle -iE_{n}\langle y_{n},u_{+}\rangle \\
									&\longrightarrow^{w}\langle y,u_{+}\rangle -iE\langle y,u_{+}\rangle\\
									&=\langle y,u_{+}\rangle_{B^{*}}.
\end{align*}
Moreover,
\begin{equation*}
\langle y_{n},u_{+}\rangle_{B^{*}}=\langle x_{n},u_{+}\rangle_{B^{*}}+\langle c_{n}u_{+},u_{+}\rangle_{B^{*}}
					+\langle c_{n}e^{-2i\theta_{n}}u_{-},u_{+}\rangle_{B^{*}}
					=2c_{n}
\end{equation*}
\begin{equation*}
\langle y,u_{+}\rangle_{B^{*}}=\langle x,u_{+}\rangle_{B^{*}}+\langle au_{+},u_{+}\rangle_{B^{*}}
					+\langle bu_{-},u_{+}\rangle_{B^{*}}
					=2a.
\end{equation*}
Therefore $c_{n}\longrightarrow a$. Then 
\begin{equation*}
x_{n}+c_{n}u_{+}+c_{n}e^{-2i\theta_{n}}u_{-}\longrightarrow^{w}x+au_{+}+ae^{-2i\theta}u_{-}.
\end{equation*}
By uniqueness of limits $y=x+au_{+}+ae^{-2i\theta}u_{-}$. Hence
$y\in Ker(T_{\theta}-EI)\setminus\left\lbrace 0\right\rbrace$. Finally, $Q_{M}$ is closed.
\end{proof}

Let us mention the next identity.

\begin{lemma}\label{lem3}
Let $y_{j}=x_{j}+c_{j}e^{i\theta_{j}}u_{+}+c_{j}e^{-i\theta_{j}}u_{-}\in Ker(T_{\theta_{j}}-E_{j}I)$ with $j=1,2$.
Then
\begin{equation*}
-4c_{1}\overline{c_{2}}sen(\theta_{1}-\theta_{2})=(E_{1}-E_{2})\langle y_{1},y_{2}\rangle.
\end{equation*}
\end{lemma}

The following result is formulated like in \cite{Gordon}.

\begin{lemma}\label{funcionW}
Let $F\subseteq Q_{M}$ be a compact set. The function $W_{F}:F\times F\longrightarrow\C$ such that
\begin{equation*}
W_{F}\left((\theta_{1},E_{1},y_{1}),(\theta_{2},E_{2},y_{2})\right) :=\langle y_{1},y_{2}\rangle
\end{equation*}
is continuous at least at a pair $(\varepsilon_{0},\varepsilon_{0})\in F\times F$.
\end{lemma}
\begin{proof}
Let $\left\lbrace e_{n}\right\rbrace_{n\in\N} $ be a orthonormal basis of $\M$. We define
$J_{m},J_{F}:F\longrightarrow\R$  with $m\in\N$ and $\varepsilon=(\theta,E,y)$ by
\begin{equation*}
J_{m}(\varepsilon):=\sum_{n=1}^{m}\vert\langle y, e_{n}\rangle\vert^{2}
\end{equation*}
and
\begin{equation*}
J_{F}(\varepsilon):=\Vert y\Vert^{2}.
\end{equation*}
Due to Parseval's identity, it turns out that for each $\varepsilon\in F$
\begin{equation*}
J_{F}(\varepsilon)=\sum_{n=1}^{\infty}\vert\langle y,e_{n}\rangle\vert^{2}
=\lim_{m\rightarrow\infty}J_{m}(\varepsilon).
\end{equation*}
Let $\varepsilon_{k}=(\theta_{k},E_{k},y_{k})\in F$ be a sequence that converges to $\varepsilon\in F$. Then,
$y_{k}\longrightarrow^{w}y$. Hence,
\begin{equation*}
\lim_{k\rightarrow\infty}J_{m}(\varepsilon_{k})
=\sum_{n=1}^{m}\vert\langle y,e_{n}\rangle\vert^{2}=J_{m}(\varepsilon).
\end{equation*}
Therefore, $J_{F}$ is pointwise limit of continuous functions. By \cite[Theorem 7.3]{Oxtoby}, there is some
$\varepsilon_{0}\in F$ such that $J_{F}$ is continuous at $\varepsilon_{0}=(\theta_{c},E_{c},y_{0})$.\\

We assert that $W_{F}$ is
continuous at $(\varepsilon_{0},\varepsilon_{0})$. If $\varepsilon_{n}=(\theta_{n},E_{n},y_{n})\longrightarrow\varepsilon_{0}$
and $\varepsilon'_{n}=(\theta'_{n},E'_{n},y'_{n})\longrightarrow\varepsilon_{0}$, then
\begin{equation*}
y_{n},y'_{n}\longrightarrow^{w}y_{0}
\end{equation*}
and by continuity of $J_{F}$ at $\varepsilon_{0}$,
\begin{equation*}
\Vert y_{n}\Vert,\Vert y'_{n}\Vert\longrightarrow \Vert y_{0}\Vert.
\end{equation*}
So, $y_{n},y'_{n}\longrightarrow y_{0}$.
Thus, 
\begin{equation*}
W_{F}(\varepsilon_{n},\varepsilon'_{n})=\langle y_{n},y'_{n}\rangle
\longrightarrow\langle y_{0}, y_{0}\rangle=W_{F}(\varepsilon_{0},\varepsilon_{0}).
\end{equation*}
Finally, $W_{F}$ is continuous at $(\varepsilon_{0},\varepsilon_{0})$.
\end{proof}

Let $\P:\tau\longrightarrow\R$, $\Pi:\tau\longrightarrow\left[ 0,\pi\right]$, $\RR:\tau\longrightarrow\left[ 0,\pi\right]\times\R$, $p:\left[ 0,\pi\right]\times\R\longrightarrow\R$ and $q:\left[ 0,\pi\right]\times\R\longrightarrow\left[ 0,\pi\right]$ be projections. We follow the same procedure that is used in the proof of \cite[Proposition 2*]{Gordon}.

\begin{prop}\label{prop2*}
If $F\subseteq Q_{M}$ is compact and $\P(F)$ is nowhere dense if and only if $\Pi(F)$ is nowhere dense.
\end{prop}
\begin{proof}
Suppose that there is a compact subset $F\subseteq Q_{M}$ such that $\P(F)$ is nowhere dense and $\Pi(F)$ contains a non-empty open set of $\left[ 0,\pi \right]$ which one denotes as $I$. Consider the partially ordered set given by the collection
\begin{equation*}
X:=\left\lbrace F'\subseteq F: F' \;is\;compact\;set\;in\;Q_{M}\;and\;\Pi(F')\supset I\right\rbrace.
\end{equation*}
We assert that $X$ has a minimal element. Let $\left\lbrace F_{\alpha}\right\rbrace_{\alpha\in\Delta} $ be
a chain in $X$ where $\Delta$ is an arbitrary set of indices. It turns out that
$\bigcap_{\alpha\in\Delta}F_{\alpha}$ is a compact subset of $F$ in $Q_{M}$. We assert that
$\Pi\left( \bigcap_{\alpha\in\Delta}F_{\alpha}\right)\supset I $. For all $\alpha\in\Delta$, there exists
$G_{\alpha}\subseteq\R\times\M$ such that $F_{\alpha}=\Pi(F_{\alpha})\times G_{\alpha}$. The inclusion follows
from the following fact
\begin{equation*}
\bigcap_{\alpha\in\Delta}F_{\alpha}\supseteq\bigcap_{\alpha\in\Delta}\Pi(F_{\alpha})\times\bigcap_{\alpha\in\Delta}G_{\alpha}.
\end{equation*}
Therefore, $\bigcap_{\alpha\in\Delta}F_{\alpha}$ is a lower bound of
$\left\lbrace F_{\alpha}\right\rbrace_{\alpha\in\Delta} $. We conclude by Zorn's lemma. Denote the minimal set by $\FF$.\\

Now we want to prove that there is some subset of $\FF$ such that its projections with respect to $\P$ and $\Pi$ are homeomorphic.
By Lemma \ref{funcionW}, there exists $\delta >0$ such that for all $\varepsilon,\varepsilon'\in \FF$
\begin{equation}\label{qqq}
d_{\FF\times \FF}\left( (\varepsilon,\varepsilon'),(\varepsilon_{0},\varepsilon_{0})\right)<\delta
\Rightarrow\left\vert W_{\FF}(\varepsilon,\varepsilon')-W_{\FF}(\varepsilon_{0},\varepsilon_{0})\right\vert <
\dfrac{\Vert y_{0}\Vert^{2}}{2}
\end{equation}
where $d_{\FF\times \FF}$ is the metric of $\FF\times \FF$.\\

Let $U$ be the ball in $\tau_{M}$ defined as
\begin{equation*}
U:=\left\lbrace\varepsilon\in \tau_{M} :d_{M}\left( \varepsilon,\varepsilon_{0}\right)<\frac{\delta}{3}\right\rbrace 
\end{equation*}
where $d_{M}$ denotes the metric on $\tau_{M}$. We are going to show that $p\vert_{\RR(\FF\cap\overline{U})}$ and $q\vert_{\RR(\FF\cap\overline{U})}$ are injective. That is,
for all $\varepsilon_{1}=(\theta_{1},E_{1},y_{1}),\varepsilon_{2}=(\theta_{2},E_{2},y_{2})\in \FF\cap\overline{U}$, it is sufficient to prove that $\theta_{1}=\theta_{2}$ if and only if $E_{1}=E_{2}$.
\begin{itemize}
\item $(\Rightarrow)$ Suppose that $\theta_{1}=\theta_{2}$. By Lemma \ref{lem3}
\begin{equation}\label{0}
(E_{1}-E_{2})\langle y_{1},y_{2}\rangle=0.
\end{equation}
Since $\varepsilon_{1},\varepsilon_{2}\in \overline{U}$,
\begin{equation*}
d_{\FF\times \FF}\left( (\varepsilon_{1},\varepsilon_{2}),(\varepsilon_{0},\varepsilon_{0})\right)=
d_{M}\left( \varepsilon_{1},\varepsilon_{0}\right)+d_{M}\left( \varepsilon_{2},\varepsilon_{0}\right)<\delta .
\end{equation*}
By (\ref{qqq}), 
\begin{equation*}
\left\vert\langle y_{1},y_{2}\rangle-\Vert y_{0}\Vert^{2}\right\vert
=\left\vert W_{\FF}(\varepsilon_{1},\varepsilon_{2})-W_{\FF}(\varepsilon_{0},\varepsilon_{0})\right\vert
<\dfrac{\Vert y_{0}\Vert^{2}}{2}.
\end{equation*}
Thus, $\langle y_{1},y_{2}\rangle\neq 0$ implies $E_{1}=E_{2}$.

\item $(\Leftarrow)$ Suppose that $E_{1}=E_{2}$. If $\theta_{1}\neq\theta_{2}$, then the operators
$T_{\theta_{1}}\upharpoonright_{\M}$ and $T_{\theta_{2}}\upharpoonright_{\M}$ have
an eigenvalue in common, but that contradicts to Theorem \ref{teo1.1}.
\end{itemize}
Next, $p\vert_{\RR(\FF\cap\overline{U})}$ and $q\vert_{\RR(\FF\cap\overline{U})}$ are injective and therefore
homeomorphisms. Then 
\begin{center}
$\P(\FF\cap\overline{U})=p\left[ \RR(\FF\cap\overline{U})\right]$ and $\Pi(\FF\cap\overline{U})=q\left[ \RR(\FF\cap\overline{U})\right] $
\end{center}
are homeomorphic to $\RR(\FF\cap\overline{U})$ and therefore between them.\\

On the other hand, $\Pi(\FF)= \Pi(\FF\cap U)\cup \Pi(\FF\setminus U)$ where $\Pi(\FF)$ contains to $I$. Since
$\P(\FF\cap\overline{U})\subseteq \P(\FF)$, by hypothesis it is nowhere dense in $\R$. But
$\Pi(\FF\cap\overline{U})$ is nowhere dense in $\left[ 0,\pi \right]$ and $\Pi(\FF\cap U)$ inherits that property. Then $I$ is not
contained in $\Pi(\FF\cap U)$. Hence $I\subseteq \Pi(\FF\setminus U)$. Moreover, $\FF\setminus U$ is properly contained in $\FF$ and is compact in $Q_{M}$. But this contradicts the minimality of $\FF$. Consequently,
$\Pi(F)$ is nowhere dense in $\left[ 0,\pi \right]$. The other direction is analogous.
\end{proof}

Finally we arrive at the following proposition.

\begin{prop}\label{prop2}
Let $Y$ be a countable union of closed nowhere dense subsets in $\R$. Then
\begin{equation*}
\left\lbrace \theta\in\left[ 0,\pi\right)\mid\sigma_{p}(T_{\theta}\upharpoonright_{\M})\cap Y\neq\varnothing\right\rbrace 
\end{equation*}
is a countable union of closed nowhere dense subsets in $\left[ 0,\pi\right]$.
\end{prop}
\begin{proof}
It turns out that
\begin{equation*}
\left\lbrace \theta\in\left[ 0,\pi\right)\mid\sigma_{p}(T_{\theta}\upharpoonright_{\M})\cap Y\neq\varnothing\right\rbrace
=\Pi(Q\cap \P^{-1}(Y))
\end{equation*}
where $Q:=\left\lbrace (\theta,E,y)\in\tau:y\in Ker(T_{\theta}-EI)\;and\;\langle y,u_{+}\rangle=1\right\rbrace$.\\

By hypothesis, there is a sequence $\left\lbrace Y_{n}\right\rbrace _{n\in\N} $ of closed nowhere dense subsets
in $\R$ such that $Y= \bigcup_{n\in\N}Y_{n}$. For all $M\in\N$ we define
\begin{center}
$Q^{(M)}:=Q_{M}\cap  \left( \left[ -M,M\right]\times\left[-M,M\right]\times B_{M}\cap\M\right)  $ 
\end{center}
which is closed in $\tau_{M}$ by Lemma \ref{lem2}. Further, due to Tychonoff Theorem and Banach-Alaoglu-Bourbaki Theorem,
$Q^{(M)}$ is compact in $\tau_{M}$. It is true that $Q=\bigcup_{M\in\N}Q^{(M)}$. Then,
\begin{equation}\label{proyecciones}
\Pi(Q\cap \P^{-1}(Y))=\bigcup_{M,n\in\N}\Pi\left( Q^{(M)}\bigcap \P^{-1}\left( Y_{n}\right) \right).
\end{equation}
Since $Q^{(M)}\bigcap \P^{-1}\left( Y_{n}\right)$ is closed and contained in a compact, it inherits compactness. Furthermore, $\P\left( Q^{(M)}\bigcap \P^{-1}\left( Y_{n}\right)\right)=\P\left( Q^{(M)} \right) \bigcap Y_{n} $ is contained in a nowhere dense set, thus, it inherits that property. According to the last proposition, $\Pi\left( Q^{(M)}\bigcap \P^{-1}\left( Y_{n}\right) \right)$ is nowhere dense. In addition, it is compact by continuity of $\Pi$. But compact implies closed in $\R^{2}$.
\end{proof}

With the above we can prove our theorem on \textbf{forbidden extension parameters}.

\begin{theorem}\label{TEO3}
Let $\theta_{0}$ fixed. Then
\begin{equation}\label{setprincipal}
\left\lbrace \theta\in\left[ 0,\pi\right]\mid\sigma_{p}(T_{\theta}\upharpoonright_{\M})\cap\sigma(T_{\theta_{0}}\upharpoonright_{\M})=\varnothing\right\rbrace 
\end{equation}

is dense $\textit{G}_{\delta}$ in $\left[ 0,\pi\right]$.
\end{theorem}
\begin{proof}
Consider the set
\begin{equation*}
Y:=\sigma(T_{\theta_{0}}\upharpoonright_{\M})\cap\bigcup_{\theta\in\left[ 0,\pi\right)\setminus\lbrace \theta_{0}\rbrace} \sigma_{p}(T_{\theta}\upharpoonright_{\M}).
\end{equation*}
By Proposition \ref{prop2} and Corollary \ref{prop1} one has that
\begin{equation*}
\left\lbrace \theta\in\left[ 0,\pi\right]\mid\sigma_{p}(T_{\theta}\upharpoonright_{\M})\cap Y\neq\varnothing\right\rbrace
\end{equation*}
is a countable union of closed nowhere dense subsets in $\left[ 0,\pi\right]$. Thus,
\begin{equation*}
N:=\left\lbrace \theta\in\left[ 0,\pi\right]\mid\sigma_{p}(T_{\theta}\upharpoonright_{\M})\cap Y=\varnothing\right\rbrace
\end{equation*}
is dense $\textit{G}_{\delta}$ in $\left[ 0,\pi\right]$. Furthermore

\begin{equation*}\sigma_{p}(T_{\theta}\upharpoonright_{\M})\bigcap Y=\left\{ \begin{array}{lcc}
                				\sigma_{p}(T_{\theta}\upharpoonright_{\M})\bigcap\sigma(T_{\theta_{0}}\upharpoonright_{\M})	&   if  & \theta\neq\theta_{0} \\
       
             					\\ \varnothing &  if  &\theta =\theta_{0}
             			\end{array}
             			 \right.\end{equation*}
Therefore, $\theta_{0}\in N$. However, 
\begin{center}
$\theta_{0}$ belongs to the set (\ref{setprincipal}) if and only if
$\sigma_{p}(T_{\theta_{0}}\upharpoonright_{\M})=\varnothing$.
\end{center}
 Then

\begin{equation*}
N =\left\lbrace \theta\in\left[ 0,\pi\right]\mid\sigma_{p}(T_{\theta}\upharpoonright_{\M})\cap\sigma(T_{\theta_{0}}\upharpoonright_{\M})=\varnothing\right\rbrace\cup\left\lbrace \theta_{0}\right\rbrace.
\end{equation*}
\begin{itemize}
\item Case 1. If $\sigma_{p}(T_{\theta_{0}}\upharpoonright_{\M})=\varnothing$, then $N$ is equal to the set
(\ref{setprincipal}).
\item Case 2. Suppose $\sigma_{p}(T_{\theta_{0}}\upharpoonright_{\M})\neq\varnothing$. Thus,
\begin{equation*}
N\setminus\left\lbrace \theta_{0}\right\rbrace =\left\lbrace \theta\in\left[ 0,\pi\right]\mid\sigma_{p}(T_{\theta}\upharpoonright_{\M})\cap\sigma(T_{\theta_{0}}\upharpoonright_{\M})=\varnothing\right\rbrace.
\end{equation*}
Since $N$ is dense in $\left[ 0,\pi\right]$, $N\setminus\left\lbrace \theta_{0}\right\rbrace$ is too.
Moreover, if $\left\lbrace F_{n}\right\rbrace _{n\in\N} $ is the sequence of open subsets in $\left[ 0,\pi\right]$ such that $N=\bigcap_{n\in\N}F_{n}$, then
$N\setminus\left\lbrace \theta_{0}\right\rbrace=\bigcap_{n\in\N}\left( F_{n}\setminus\left\lbrace \theta_{0}\right\rbrace\right) $ and each $F_{n}\setminus\left\lbrace \theta_{0}\right\rbrace$ is open in $\left[ 0,\pi\right]$.
\end{itemize}
In conclusion,
\begin{equation*}
\left\lbrace \theta\in\left[ 0,\pi\right]\mid\sigma_{p}(T_{\theta}\upharpoonright_{\M})\cap\sigma(T_{\theta_{0}}\upharpoonright_{\M})=\varnothing\right\rbrace
\end{equation*}
is dense $\textit{G}_{\delta}$ in $\left[ 0,\pi\right]$.
\end{proof}

Denote by $\sigma _{ac}$ and $\sigma _{sc}$ the absolutely continuous and singular continuous spectrum respectively. We mean by $int$ to interior of a set. Remind that $\sigma _{p}$ denotes the set of eigenvalues. We get the following corollaries. 

\begin{corollary}\label{cor4.1}
Let $\theta_{0}$ fixed and suppose $\sigma_{ac}(T_{\theta_{0}}\upharpoonright_{\M})=\varnothing$.
Then 
\begin{equation*}
\left\lbrace \theta\in\left[ 0,\pi\right]\mid\sigma(T_{\theta}\upharpoonright_{\M})\cap int\;\sigma(T_{\theta_{0}}\upharpoonright_{\M})\subseteq  \sigma_{sc}(T_{\theta}\upharpoonright_{\M}) \right\rbrace 
\end{equation*}
is dense $\textit{G}_{\delta}$ in $\left[ 0,\pi\right]$.
\end{corollary}
\begin{proof}
Follows by the invariance of absolutely continuous spectrum for self-adjoint extensions and Theorem \ref{TEO3}.
\end{proof}
\begin{corollary}\label{discreteeigen}
For a dense $\textit{G}_{\delta}$ set of self-adjoint extensions of a densely defined closed simetric operator with
deficiency indices $(1,1)$, their eigenvalues are isolated.
\end{corollary}
\begin{proof}
We make use of the invariance of essential spectrum for self-adjoint extensions and Theorem \ref{TEO3}.
\end{proof}

We conclude the following proposition.

\begin{prop}
Let $Y$ be a countable union of closed nowhere dense subsets in $\left[ 0,\pi\right]$. Then
\begin{equation*}
\sigma(T_{\theta_{0}}\upharpoonright_{\M})\cap\bigcup_{\theta\in Y}\sigma_{p}(T_{\theta}\upharpoonright_{\M})
\end{equation*}
is a countable union of closed nowhere dense subsets in $\sigma(T_{\theta_{0}}\upharpoonright_{\M})$.
\end{prop}
\begin{proof}
It follows the same line that proof of Proposition \ref{prop2} by changing the roles of $\Pi$ and $\P$ in (\ref{proyecciones}).
\end{proof}

\begin{remark}\label{equivalencianota}
Theorem \ref{zzz} is equivalent to Theorem \ref{TEO3}. The first direction is
deduced by taking in Proposition \ref{prop2} the set (\ref{prop1set}) instead of set $Y$. This is just the proof
Theorem \ref{TEO3}. The converse is
deduced by taking in the previous proposition the complement of set (\ref{setprincipal}) instead of set $Y$.
\end{remark}


In order to prove the second main theorem we note the following fact.

\begin{lemma}
Let $\theta':=\frac{1}{2}arg\left( -\frac{c-i}{c+i}\right)$  with $c\in\R\setminus\left\lbrace 0\right\rbrace$. The functions
\begin{equation*}
\Psi_{c}:\R\setminus\left\lbrace -\frac{1}{c}\right\rbrace \longrightarrow\left( 0,\pi\right)\setminus\left\lbrace \theta'\right\rbrace\;and\;
\Psi:\R\longrightarrow\left( 0,\pi\right)
\end{equation*}
where
\begin{equation*}
\Psi_{c}(\alpha):=\frac{1}{2}arg\left[ -\frac{1+\alpha(c-i)}{1+\alpha(c+i)}\right]\;and\;
\Psi(\alpha):=\Psi_{c=0}(\alpha)
\end{equation*}
are homeomorphisms.
\end{lemma}
\begin{proof}
It is enough to note that 
\begin{equation*}
\Psi_{c}^{-1}: \left( 0,\pi\right)\setminus\left\lbrace \theta'\right\rbrace\longrightarrow\R\setminus\left\lbrace -\frac{1}{c}\right\rbrace\;and\;\Psi^{-1}: \left( 0,\pi\right)\longrightarrow\R
\end{equation*}
where
\begin{equation*}
\Psi_{c}^{-1}(\theta):=-\frac{1+e^{2\theta i}}{c-i+(c+i)e^{2\theta i}}\;and\;\Psi^{-1}(\theta):=\Psi_{c=0}^{-1}(\theta)
\end{equation*}
is the inverse function of $\Psi_{c}$ and $\Psi$ respectively. In addition, both $\Psi_{c}$ and $\Psi$ are composition of continuous functions. Analogously for $\Psi_{c}^{-1}$ and $\Psi^{-1}$.
\end{proof}

Finally we can prove the second main theorem.

\begin{proof}[Proof of Theorem \ref{maintheorem2}]
Suppose $c\neq 0$. By Theorem \ref{ultimate},
\begin{equation*}
-\frac{1}{c}\in\Gamma_{c}\Longleftrightarrow 0\in M
\end{equation*}
where $M$ denotes to the set (\ref{setprincipal}) with $\theta_{0}=\frac{\pi}{2}$. If $-\frac{1}{c}\not\in\Gamma_{c}$, then $\Gamma_{c}$ can be mapped by $\Psi_{c}$. Again by Theorem \ref{ultimate} and putting $\theta=\Psi_{c}(\alpha)$,

\begin{align*}
\Psi_{c}(\Gamma_{c})&=\left\lbrace \Psi_{c}(\alpha):\sigma_{p}(A_{\alpha,c}\upharpoonright_{\M})\cap\sigma(A\upharpoonright_{\M})=\varnothing\right\rbrace \\
			  &=\left\lbrace \theta\in\left( 0,\pi\right)\setminus\left\lbrace \theta'\right\rbrace :\sigma_{p}\left( T_{\theta}\upharpoonright_{\M}\right) \cap\sigma\left( T_{\frac{\pi}{2}}\upharpoonright_{\M}\right) =\varnothing\right\rbrace.
\end{align*}
It is obtained that $\Gamma_{c}=\Psi_{c}^{-1}\left( M\setminus\left\lbrace\theta'\right\rbrace\right) $. We realize of the following:
\begin{itemize}
\item[a)] If $\theta'\not\in M$, one concludes by Theorem \ref{TEO3}.
\item[b)] If $\theta'\in M$, the idea of Case 2 in Proof of Theorem \ref{TEO3} is followed for the set
$\Gamma_{c}$.
\end{itemize}
For the case in which $-\frac{1}{c}\in\Gamma_{c}$, we have that
$\Gamma_{c}\setminus\left\lbrace -\frac{1}{c}\right\rbrace =\Psi_{c}^{-1}\left( M\setminus\left\lbrace 0,\theta'\right\rbrace\right) $. Therefore, $\Gamma_{c}$ contains a dense $G_{\delta}$ subset in $\R$.\\

If $c=0$, according to (\ref{tetace}) the corresponding coupling $\alpha$ with $\theta =0=\theta'$ is $\alpha=\infty$ which
is not considered. Therefore $\Gamma=\Psi^{-1}\left( M\setminus\left\lbrace 0\right\rbrace\right)$ and one repeats the procedure in a) and b).
\end{proof}

\begin{remark}
If $\varphi\in\H_{-1}\setminus\H$, then there is a unique set
\begin{equation*}
\left\lbrace x\in\sigma(A\upharpoonright_{\M}):x\not\in\sigma_{p}(A_{\alpha}\upharpoonright_{\M}),\forall\;\alpha\in\R\setminus\left\lbrace 0\right\rbrace \right\rbrace .
\end{equation*}
In the same way there is a unique set
\begin{equation*}
\left\lbrace \alpha\in\R\mid\sigma_{p}(A_{\alpha}\upharpoonright_{\M})\cap\sigma(A\upharpoonright_{\M})=\varnothing\right\rbrace.
\end{equation*}
Even if $\varphi\in\H_{-2}\setminus\H_{-1}$ is homogeneous to $A$ (see \cite[Subsection 1.3.3]{ALBKUR}), the set is unique.
\end{remark}

We conclude the following corollaries.

\begin{corollary}
Suppose the hypothesis of Theorem \ref{maintheorem1} and $(A+iI)^{-1}\varphi$ is cyclic for $A$. The sets
\begin{equation*}
\left\lbrace x\in\sigma(A):x\not\in\sigma_{p}(A_{\alpha,c}),\forall\;\alpha\in\R\setminus\left\lbrace 0\right\rbrace \right\rbrace 
\end{equation*}
and
\begin{equation*}
\left\lbrace \alpha\in\R\mid\sigma_{p}(A_{\alpha,c})\cap\sigma(A)=\varnothing\right\rbrace
\end{equation*}
contain a dense $\textit{G}_{\delta}$ subset in $\sigma(A)$ and $\R$ respectively.
\end{corollary}

Denote by $\sigma_{dis}$ the discrete spectrum.

\begin{corollary}
Suppose the hypothesis of Theorem \ref{maintheorem1}. Then 
\begin{equation*}
\Theta_{c}:=\left\lbrace \alpha\in\R\mid\sigma_{p}(A_{\alpha,c}\upharpoonright_{\M})=\sigma_{dis}(A_{\alpha,c}\upharpoonright_{\M})\right\rbrace
\end{equation*}
contains a dense $\textit{G}_{\delta}$ in $\R$. Also, if $\sigma_{ac}(A\upharpoonright_{\M})=\varnothing$
\begin{equation*}
\Delta_{c}:=\left\lbrace \alpha\in\R\mid\sigma(A_{\alpha,c}\upharpoonright_{\M})\cap int\;\sigma(A\upharpoonright_{\M})\subseteq  \sigma_{sc}(A_{\alpha,c}\upharpoonright_{\M}) \right\rbrace 
\end{equation*}
contains a dense $\textit{G}_{\delta}$ in $\R$.
\end{corollary}
\begin{proof}
The proof follows similar lines to Corollary \ref{cor4.1} and \ref{discreteeigen}.
\end{proof}

\begin{corollary}
If the hypothesis of Theorem \ref{maintheorem1} holds and $\varphi\in\H_{-2}\setminus\H_{-1}$, then 
\begin{equation}\label{thetaset}
\left\lbrace \left( \alpha,c\right) \in\R^{2}\mid\sigma_{p}(A_{\alpha,c}\upharpoonright_{\M})=
\sigma_{dis}(A_{\alpha,c}\upharpoonright_{\M})\right\rbrace 
\end{equation}
is a dense set in $\R^{2}$. In addition if $\sigma_{ac}(A\upharpoonright_{\M})=\varnothing$, 
\begin{equation}\label{deltaset}
\left\lbrace \left( \alpha,c\right) \in\R^{2}\mid\sigma(A_{\alpha,c}\upharpoonright_{\M})\cap int\;\sigma(A\upharpoonright_{\M})\subseteq  \sigma_{sc}(A_{\alpha,c}\upharpoonright_{\M}) \right\rbrace 
\end{equation}
is a dense set in $\R^{2}$
\end{corollary}
\begin{proof}
It suffices to see that (\ref{thetaset}) and (\ref{deltaset}) are equal to $\cup_{c\in\R}\left( \Theta_{c}\times\left\lbrace c\right\rbrace \right) $ and $\cup_{c\in\R}\left( \Delta_{c}\times\left\lbrace c\right\rbrace \right) $ respectively.
\end{proof}


\begin{remark}
One concludes that just as in the case of rank one regular perturbations the absence of absolutely continuous spectrum implies the existence of singular continuous spectrum for a subfamily of rank one singular perturbations. Moreover by
Remark \ref{equivalencianota}, Theorem \ref{maintheorem1} is equivalent to Theorem \ref{maintheorem2}.
\end{remark}

\section*{Final Remarks}


In the unified approach presented here we used properties of spectral measures following \cite{RMS} and Aronszajn-Donoghue Theory to show that there is a forbidden set of energies for rank one singular perturbations. By adapting the Gordon's methods of \cite{Gordon}, we related this set to the coupling constants for such perturbations. We found out that the existence of a subset in spectrum of an unperturbed operator, which cannot contain
eigenvalues of the perturbations, is equivalent to the existence of a large set of perturbations which do not have
embedded point spectrum. In future works the unified approach presented here will be used to analyze the spectra of singular finite rank and supersingular perturbations.\\


  


\noindent\textbf{Acknowledgements}\\
\textit{The first author was supported by CONAHCYT (CVU:918585). The authors thank Luis Silva for helpful
comments.}

\bigskip

\noindent Rafael del R\'io\\ 
delrio@iimas.unam.mx\\
ORCID: \url{https://orcid.org/0000-0002-9842-6952} \bigskip 

\noindent {\small
\noindent Universidad Nacional Aut\'onoma de M\'exico\\
Instituto de Investigaciones en Matem\'aticas Aplicadas y en Sistemas\\
Departamento de F\'isica Matem\'atica\\
Ciudad de M\'exico, C.P. 04510.
}\bigskip

\noindent Mario Ruiz\\
marioruiz@comunidad.unam.mx\\
ORCID: \url{https://orcid.org/0009-0009-0837-2996} \bigskip

\noindent {\small
\noindent Universidad Nacional Aut\'onoma de M\'exico\\
Instituto de Investigaciones en Matem\'aticas Aplicadas y en Sistemas\\
Departamento de F\'isica Matem\'atica\\
Ciudad de M\'exico, C.P. 04510.
}\bigskip

\end{document}